\documentclass[a4paper,12pt]{amsart}
\usepackage{times}
\usepackage{bm}
\usepackage{amssymb}
\usepackage{latexsym}
\usepackage{amsmath}
\usepackage{amscd}
\usepackage[dvipdfmx]{graphicx}
\usepackage[dvipdfmx]{pict2e}
\usepackage{color}

\begin{document}

\title[Minimal generating sets of groups of Kim-Manturov]
{Minimal generating sets of groups of Kim-Manturov}

\author{Takuya Sakasai}
\address{Graduate School of Mathematical Sciences, 
The University of Tokyo, 
3-8-1 Komaba, 
Meguro-ku, Tokyo, 153-8914, Japan}
\email{sakasai@ms.u-tokyo.ac.jp}
\author{Yuuki Tadokoro}
\address{Faculty of Science Division II, Department of Mathematics, 
Tokyo University of Science, 1-3 Kagurazaka, Shinjuku-ku, Tokyo 162-8601, Japan}
\email{tado@rs.tus.ac.jp}
\author{Kokoro Tanaka}
\address{Department of Mathematics, 
Tokyo Gakugei University, 
4-1-1 Nukuikita-machi, 
Koganei-shi, Tokyo 184-8501, Japan}
\email{kotanaka@u-gakugei.ac.jp}

\thanks{
The authors were partially supported by KAKENHI 
(No.19H01785, No.21H00986, No.21K03239, No.21K03220), 
Japan Society for the Promotion of Science, Japan.
}



\newtheorem{thm}{Theorem}[section]
\newtheorem{prop}[thm]{Proposition}
\newtheorem{lem}[thm]{Lemma}
\newtheorem{cor}[thm]{Corollary}
\theoremstyle{definition}
\newtheorem{definition}[thm]{Definition}
\newtheorem{example}[thm]{Example}
\newtheorem{remark}[thm]{Remark}
\newtheorem{problem}[thm]{Problem}
\newtheorem{conj}[thm]{Conjecture}
\renewcommand{\theproblem}{}

\newcommand{\Ker}{\mathop{\mathrm{Ker}}\nolimits}
\newcommand{\Hom}{\mathop{\mathrm{Hom}}\nolimits}
\renewcommand{\Im}{\mathop{\mathrm{Im}}\nolimits}

\newcommand{\Der}{\mathop{\mathrm{Der}}\nolimits}
\newcommand{\Out}{\mathop{\mathrm{Out}}\nolimits}
\newcommand{\Aut}{\mathop{\mathrm{Aut}}\nolimits}
\newcommand{\End}{\mathop{\mathrm{End}}\nolimits}
\newcommand{\Q}{\mathbb{Q}}
\newcommand{\Z}{\mathbb{Z}}
\newcommand{\R}{\mathbb{R}}
\newcommand{\C}{\mathbb{C}}

\newlength{\Width}
\newlength{\Height}
\newlength{\Depth}

\begin{abstract}
We consider a series of groups defined by Kim and Manturov. These groups have 
their background in triangulations of a surface and configurations of points, lines or circles 
on the surface. They are expected to have relationships to many geometric objects. 
In this paper, we give a minimal generating set of the group and determine 
the abelianization. We also introduce some related groups 
which might be helpful to understand the structure of the original groups. 
\end{abstract}

\renewcommand\baselinestretch{1.1}
\setlength{\baselineskip}{16pt}

\newcounter{fig}
\setcounter{fig}{0}

\maketitle

\section{Introduction}\label{sec:intro}
In the paper \cite{KM}, Kim and Manturov defined a series of groups $\Gamma_n^4$ 
given by explicit presentations. 
We set $[n]=\{1,2,\ldots, n\}$ for an integer $n \ge 4$. 
The group $\Gamma_n^4$ is generated by the symbols 
$(ijkl)$ for an ordered quadruple of four distinct integers $i,j,k,l \in [n]$. 
Here we write $(ijkl)$ for $d_{ijkl}$ in \cite{KM} for visibility. 
The defining presentation for $\Gamma_n^4$ is as follows.
\begin{definition}\label{def:gamma}
For $n \ge 4$, the group $\Gamma_n^4$ is defined by the following presentation: 

(Generators) $\{(ijkl) \mid \{i,j,k,l\}\subset [n],\ (i,j,k,l\text{: distinct})\}$

(Relations) There are four types of relations: 
\[\begin{array}{ll}
 (1)&  (ijkl)^2 =1; \\
 (2)&  (ijkl)(stuv)=(stuv)(ijkl),\ (|\{i,j,k,l\} \cap \{s,t,u,v\}| \le 2);\\
 (3)&  (ijkl)(ijlm)(jklm)(ijkm)(iklm)=1, \ (i,j,k,l,m \text{ distinct});\\
 (4)&  (ijkl)=(jkli)=(lkji).
\end{array}\]
\end{definition}
We call the relations (1) the {\it involutive relations}, (2) the {\it commutative relations}, 
(3) the {\it pentagon relations} and (4) the {\it dihedral relations}. 
Specifically, we call (3) for fixed $i,j,k,l,m$ the {\it pentagon relation for $\{i,j,k,l,m\}$}, 
where we respect the order of $i,j,k,l,m$. That is, the pentagon relation for $\{j,i,k,l,m\}$ is 
different from that for $\{i,j,k,l,m\}$ for instance. 

The background of the group $\Gamma_n^4$ is explained in the paper \cite{KM} and 
the book \cite{M_book}, where they derive the above presentation 
from some observations on configurations of points and 
triangulations of a surface. 
Indeed, the above relations are obtained from relations among Whitehead moves for 
triangulations as in Figure \ref{fig:geom}, which have their origin in 
the well-known theory of the ideal cell decomposition of the decorated Teichm\"uller space 
(see Penner \cite{Penner} for example). 
However, we should recognize that when we consider the group $\Gamma_n^4$, 
geometric objects like points, lines, triangulated surfaces etc., are unnecessary. 
The group $\Gamma_n^4$ itself stands as a highly abstract object. 
Then the following questions naturally arise: 
to what extent does this abstract group 
capture real geometric properties, and does it exhibit an interesting structure as a group? 
In any case, the paper \cite{KM} and the book \cite[Chapter 15]{M_book} discuss possible relationships 
to other geometrical objects. 
For example, a homomorphism from the pure braid group $P_n$ of $n$ strings 
to $\Gamma_n^4$ is constructed. 

\begin{figure}[htbp]
\begin{center}
\includegraphics[width=0.8\textwidth]{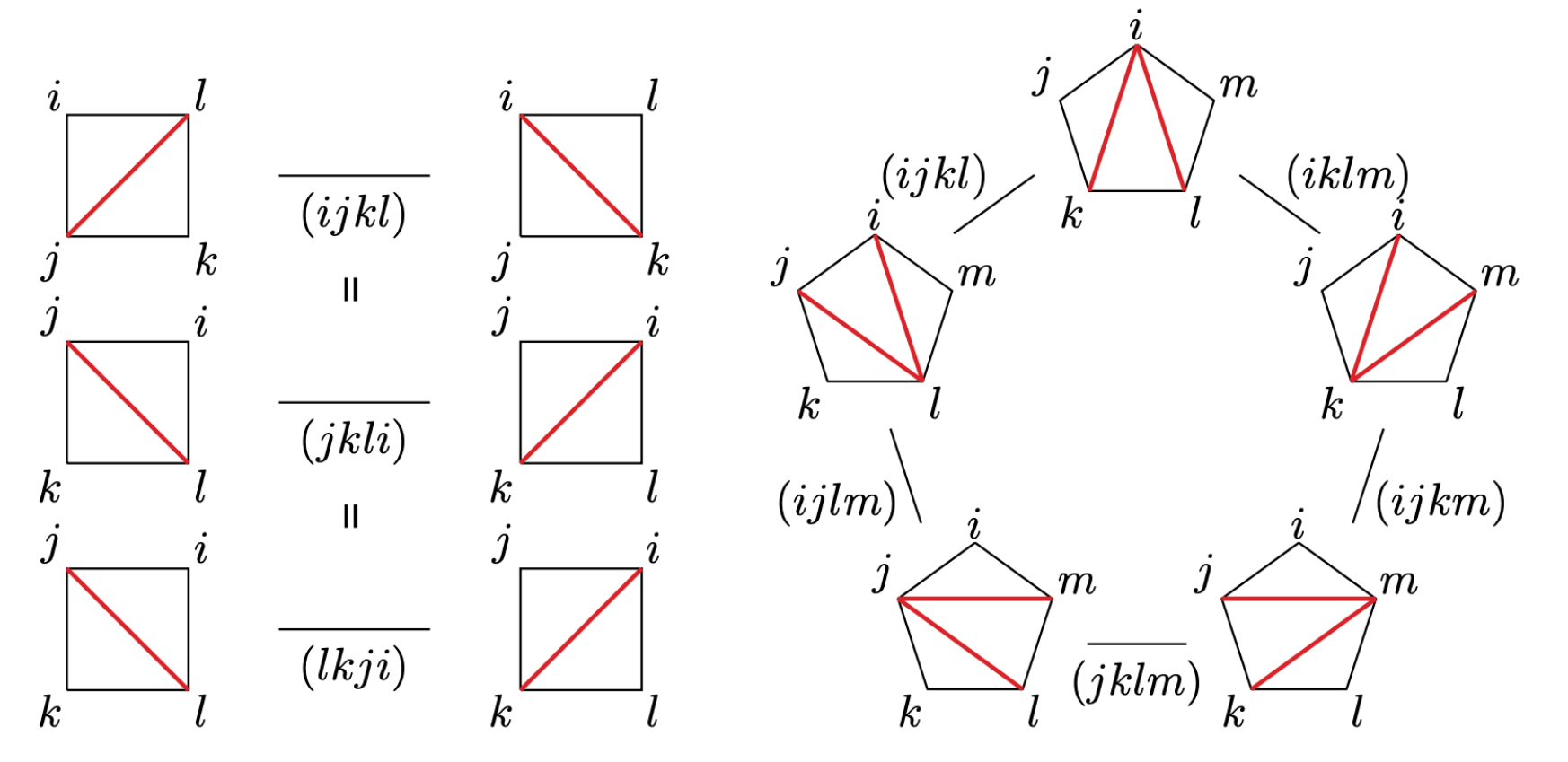}
\caption{Graphical meaning of the relations of $\Gamma_n^4$}
\label{fig:geom}
\end{center}
\end{figure}

\noindent
On the other hand, the structure of the group $\Gamma_n^4$ itself has not yet been studied. 
As long as the authors checked, even the non-triviality of $\Gamma_n^4$ for general $n$ is not given 
in a written form, although this fact is not difficult to see it. 
The purpose of this paper is to attack this issue from a purely group theoretical point of view. 
Indeed, our starting point is the presentation of Definition \ref{def:gamma}. 
The main results include to give a minimal generating set and determine 
the abelianization $H_1 (\Gamma_n^4)$ 
of $\Gamma_n^4$. We will see that $H_1 (\Gamma_n^4)\neq 0$ for all $n \ge4$, 
which directly implies the non-triviality of $\Gamma_n^4$. 

The contents of this paper is as follows. 
In Section \ref{sec:gammahat}, 
we introduce another series of groups denoted as $\widehat{\Gamma_n^4}$, which 
have a relationship with the groups $\Gamma_n^4$ 
as that between Artin groups and Coxeter groups. 
We obtain their minimal generating sets and determine their abelianizations 
ahead of these tasks for $\Gamma_n^4$, which are completed in Section \ref{sec:gamma}. 
Then we take a detour in Section \ref{sec:rep}, where we interpret our computation 
of the complex abelianization $H_1 (\widehat{\Gamma_n^4};\mathbb{C})$ from a 
representation theoretical point of view. 
This interpretation might hold significance for further studies.  
In Section \ref{sec:n5}, we focus on the case where $n=5$. 
We see that the group $\Gamma_5^4$ is an infinite non-commutative group. 
We provide two distinct proofs: one uses the program GAP and 
the other is by hand.
Finally, in Section \ref{sec:delta}, we introduce yet another sequence of groups 
denoted as $\Delta_n^4$, 
which the authors expect to be helpful to study the structure of $\Gamma_n^4$. 
We prove some fundamental facts concerning $\Delta_n^4$. 

In a forthcoming paper, we discuss the structure of $\Gamma_n^4$ for $n \ge 6$. 

\section{Minimal generating set of $\widehat{\Gamma_n^4}$}\label{sec:gammahat}

If we remind the relationship between the braid group and the permutation group, 
more generally Artin groups and Coxeter groups, it would be natural to 
introduce the following groups. 

\begin{definition}\label{def:gammahat}
For $n \ge 4$, the group $\widehat{\Gamma_n^4}$ is defined by the following presentation: 

(Generators) $\{(ijkl) \mid \{i,j,k,l\}\subset [n],\ (i,j,k,l\text{: distinct})\}$

(Relations) There are three types of relations: 
\[\begin{array}{ll}
 (2)&  (ijkl)(stuv)=(stuv)(ijkl),\ (|\{i,j,k,l\} \cap \{s,t,u,v\}| \le 2);\\
 (3)' &  (ijkl)(ijlm)(jklm)(ijkm)^{-1}(iklm)^{-1}=1, \ (i,j,k,l,m \text{ distinct});\\
 (4)' &  (ijkl)=(jkli)^{-1}=(lkji)^{-1}.
\end{array}\]
\end{definition}

We call the relation $(3)'$ the {\it signed pentagon relation} for $\{i,j,k,l,m\}$ and 
call the relations $(4)'$ the {\it signed dihedral relation}. 
The background of the above relations comes from Figure \ref{fig:geom2}. 

\begin{figure}[htbp]
\begin{center}
\includegraphics[width=0.75\textwidth]{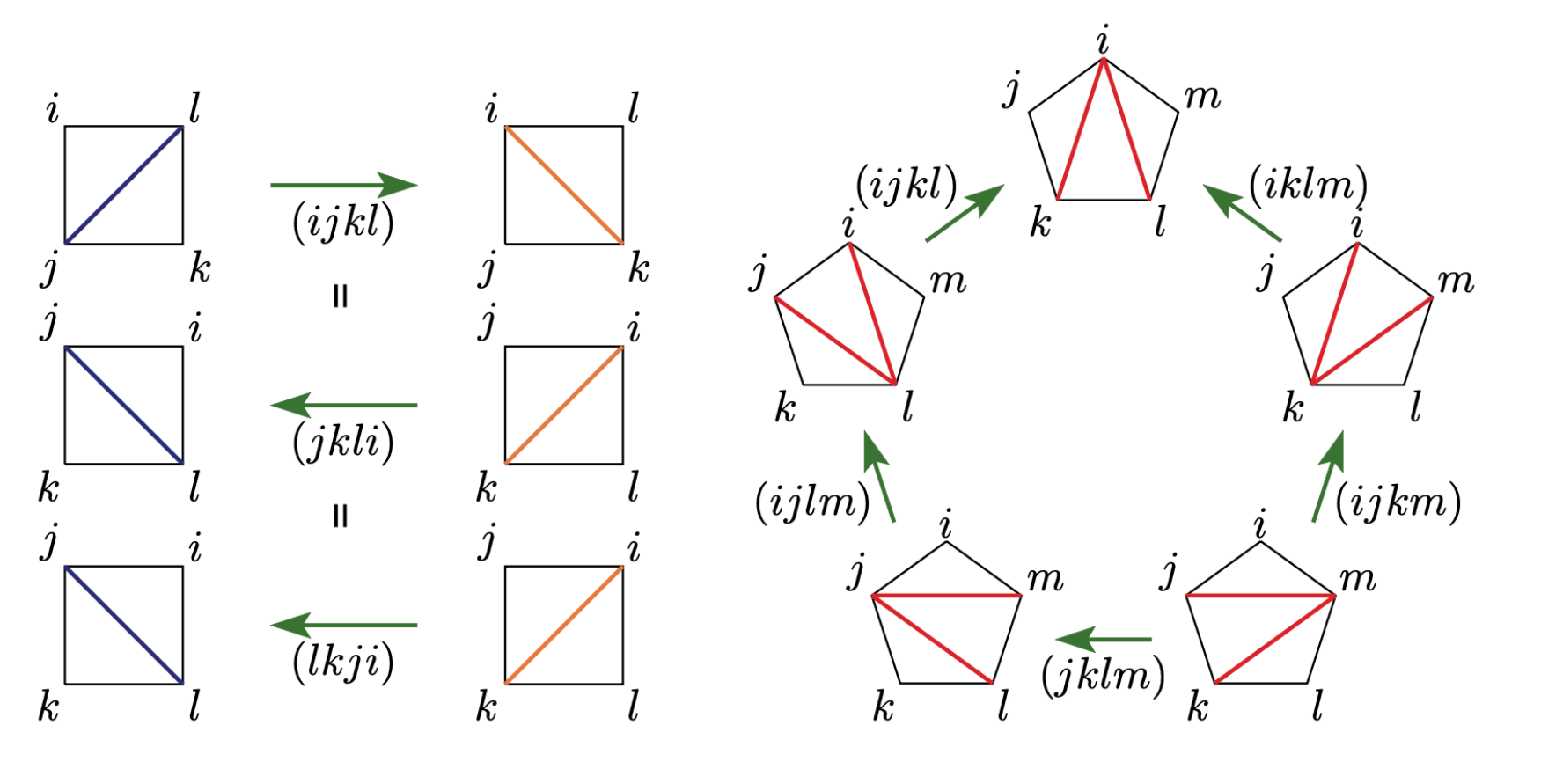}
\caption{Graphical meaning of the relations of $\widehat{\Gamma_n^4}$}
\label{fig:geom2}
\end{center}
\end{figure}

\noindent

We have a natural projection $\widehat{\Gamma_n^4} \twoheadrightarrow \Gamma_n^4$ 
sending $(ijkl) \in \widehat{\Gamma_n^4}$ to $(ijkl) \in \Gamma_n^4$.  
When $n=4$, the group $\widehat{\Gamma_4^4}$ is generated by 
$\{(1234), (1324), (1243)\}$ with no relations. 
Hence $\widehat{\Gamma_4^4} \cong \mathbb{Z}^{\ast 3}$, a free group of rank $3$. 
For $n \ge 5$, we now discuss generating sets of $\widehat{\Gamma_n^4}$. 
\begin{thm}\label{thm:generating_gammahat}
For $n \ge 4$, the group $\widehat{\Gamma_n^4}$ is generated by the set $\Lambda$ consisting of 
\begin{itemize}
\item[(G1)] $(1 2 3 k)$ with $4 \le k \le n$, 

\item[(G2)] $(1 i 2 k)$ with $3 \le i < k \le n$, 

\item[(G3)] $(1 i j k)$ with $2 \le i < k < j \le n$. 
\end{itemize}
Here, there are totally $(n-3)$, $\begin{pmatrix} n-2 \\ 2\end{pmatrix}$, 
$\begin{pmatrix} n-1 \\ 3\end{pmatrix}$ elements of 
Types {\rm (G1), (G2), (G3)}, respectively. Therefore 
$\widehat{\Gamma_n^4}$ is generated by 
\[N_n:=(n-3)+\begin{pmatrix} n-2 \\ 2\end{pmatrix}
+\begin{pmatrix} n-1 \\ 3\end{pmatrix}=\frac{(n-3)(n^2+2)}{6}=
\begin{pmatrix} n \\ 3\end{pmatrix}-1\]
elements. 
\end{thm}

\begin{proof}
When $n=4$, the statement says that $\widehat{\Gamma_n^4}$ is generated by 
$(1234), (1324), (1243)$. It is clearly true. 

We now assume that $n \ge 5$. The signed pentagon relation for $\{1,i,j,k,l\}$ says that 
\[(i j k l)=(1 i k l)^{-1} (1 i j k)^{-1}(1 j k l)(1 i j l).\]
Together with the signed dihedral relation, we see that 
\[\Upsilon:=\{ (1 a b c) \mid a < c\}\]
is a generating set of $\widehat{\Gamma_n^4}$. Note that $\Lambda \subset \Upsilon$. 
We now show that if $(1 a b c) \in \Upsilon$ is not in $\Lambda$, 
then it is written as the product of elements of $\Lambda$. 
That is, such an element is removable from $\Upsilon$ to generate 
$\widehat{\Gamma_n^4}$. 
An element of $\Upsilon$ which is {\it not} in $\Lambda$ satisfies just one of the following: 
\begin{itemize}
\item[(I)] $(1 i j k)$ with $3 \le i < j < k \le n$, 

\item[(II)] $(1 2 j k)$ with $4 \le j < k \le n$, 

\item[(III)] $(1 i j k )$ with $3 \le j < i < k \le n$.  
\end{itemize}

For an element $(1 i j k)$ of the case (I), we consider the signed pentagon relations
\[(1ijk)(1ik2)(ijk2)(1ij2)^{-1}(1jk2)^{-1}=1\]
for $\{ 1,  i,  j, k, 2\}$ and 
\[(i2kj)(i 2 j 1)(2 k j 1)(i 2 k 1)^{-1}(i k j 1)^{-1}=1\]
for $\{ i, 2, k, j, 1\}$. From these relations, we have 
\begin{align*}
(1 i j k) &= (1jk2)(1ij2)(ijk2)^{-1}(1ik2)^{-1}\\
&=(12kj)(12ji)\underline{(i2kj)^{-1}}(12ki)^{-1}\\
&=(12kj)(12ji)(i 2 j 1)(2 k j 1)(i 2 k 1)^{-1}(i k j 1)^{-1}(12ki)^{-1}\\
&=(12kj)(12ji)(1i2j)^{-1}(12kj)^{-1}(1i2k)(1ikj)(12ki)^{-1}. 
\end{align*}
The last expression consists of elements in $\Lambda$. 

For an element $(12 j k)$ of the case (II), we consider the signed pentagon relations 
\[(1 2 j k) (1 2 k 3) (2 j k 3) (1 2 j 3)^{-1} (1 j k 3)^{-1}=1 \]
for $\{ 1,2, j, k, 3\}$ and 
\[(2 1 j k) (2 1 k 3) (1 j k 3) (2 1 j 3)^{-1} (2 j k 3)^{-1}=1 \]
for $\{2, 1, j ,k ,3\}$. From these relations, we have 
\begin{align*}
(1 2 j k) &= (1 j k 3)(12j3) \underline{(2jk3)^{-1}} (12k3)^{-1} \\
&= (1 j k 3)(12j3) (2 1 j 3)(1 j k 3)^{-1}(2 1 k 3)^{-1}(2 1 j k)^{-1}(12k3)^{-1} \\
&= (13kj)(12j3)(123j)^{-1}(13kj)^{-1}(123k)(12kj)(12k3)^{-1}.
\end{align*}
The last expression consists of elements in $\Lambda$. 

For an element $(1 i j k)$ of the case (III), we consider the signed pentagon relations
\[(1 i j k) (1 i k 2) (i j k 2)(1 i j 2)^{-1} (1 j k 2)^{-1}=1 \]
for $\{ 1, i, j, k, 2\}$ and  
\[(i 2 k j) (i 2 j 1) (2 k j 1) (i 2 k 1)^{-1} (i k j 1)^{-1}=1 \]
for $\{i, 2, k, j, 1\}$. From these relations, we have 
\begin{align*}
(1 i j k) &= (1 j k 2)(1 i j 2)(i j k 2)^{-1} (1 i k 2)^{-1}\\
&=(1j k2)(1i j 2)\underline{(i 2 k j)^{-1}} (1ik2)^{-1}\\
&=(1j k2)(1i j 2) (i 2 j 1) (2 k j 1) (i 2 k 1)^{-1} (i k j 1)^{-1}(1ik2)^{-1}\\
&=(12kj)\underline{(12ji)}(1j2i)^{-1}(12kj)^{-1}(1i2k)(1jki)(12ki)^{-1}. 
\end{align*}
If $j=3$, the last expression consists of elements in $\Lambda$. Otherwise,  
we use the equality in the case (II) to get 
\begin{align*}
(1 i j k) &= (12kj)(13ij)(12j3)(123j)^{-1}(13ij)^{-1}(123i)(12ij)(12i3)^{-1} \\
& \quad \cdot (1j2i)^{-1}(12kj)^{-1}(1i2k)(1jki)(12ki)^{-1}. 
\end{align*}
The last expression consists of elements in $\Lambda$. 
This completes the proof. 
\end{proof}

Theorem \ref{thm:generating_gammahat} says that there exists 
a  surjective homomorphism 
\[\mathbb{Z}^{\ast N_n} \twoheadrightarrow \widehat{\Gamma_n^4}\]
for $n \ge 4$. Passing to their abelianizations, we have a surjective homomorphism 
\[\mathbb{Z}^{N_n} \twoheadrightarrow  H_1 (\widehat{\Gamma_n^4}).\]
We will see that the last surjection is an isomorphism. 

Let $[n]_k$ denote the set 
of $k$ elements subsets of $[n]=\{1,2,\ldots, n\}$. 
We denote by $\mathbb{Z} [n]_k$ the free abelian group based 
by the set $[n]_k$.  
Consider the homomorphism
\[\Phi_3 \colon \widehat{\Gamma_n^4}  \longrightarrow \mathbb{Z} [n]_3\]
given by 
\[\Phi_3 ((ijkl)) = \{i, j, k\} - \{i, j, l\}+\{i,k,l\}-\{j, k, l\},\]
which is well-defined.

\begin{thm}\label{thm:abelianization_gammahat}
For $n \ge 4$, the image of the homomorphism 
$\Phi_3 \colon \widehat{\Gamma_n^4} \to \mathbb{Z} [n]_3 \cong 
\mathbb{Z}^{N_n+1}$ is 
isomorphic to $\mathbb{Z}^{N_n}$. 
In fact, the image is not a direct summand in $\mathbb{Z} [n]_3$. 
Combining with Theorem $\ref{thm:generating_gammahat}$, 
we have 
\[H_1 (\widehat{\Gamma_n^4}) \cong \mathbb{Z}^{N_n}.\]
\end{thm}
\begin{proof}
When $n=4$, the group $\widehat{\Gamma_4^4}$ is 
the free group of rank $3$ generated by $(1234)$, $(1324)$ and $(1243)$. 
We can directly check that the image of $\Phi_3$ is isomorphic to $\mathbb{Z}^3$, 
which is not a direct summand in $\mathbb{Z} [4]_3 \cong \mathbb{Z}^4$. 
In fact, $\mathbb{Z} [4]_3/ \Phi_3(\widehat{\Gamma_4^4}) 
\cong \mathbb{Z} \oplus (\mathbb{Z}/2\mathbb{Z})^2$. 
%
%

Now we assume that $n \ge 5$. We endow the basis $[n]_3$ of $\mathbb{Z} [n]_3$ 
with the lexicographic order $\prec$ after writing each element of $[n]_3$ in the form 
$\{i,j,k\}$ with $i < j < k$. That is, 
\[\{1,2,3\} \prec \{1,2,4\} \prec \cdots \prec \{n-2, n-1, n\}.\]
By this total order, we regard  $[n]_3$ as an ordered basis of $\mathbb{Z} [n]_3$. 
Let us show that for each $\{i,j,k\} \in [n]_3$ except $\{n-2, n-1, n\}$, there exists 
an element $w$ of $\widehat{\Gamma_n^4}$ satisfying
\begin{itemize}
\item \ The coefficient of $\{i, j, k\}$ in $\Phi_3 (w)$ is non-zero, 
\item \ If $\{i', j', k'\} \prec \{i, j, k\}$ and $\{i', j', k'\} \neq \{i, j, k\}$, 
the coefficient of $\{i', j', k'\}$ in $\Phi_3 (w)$ is zero. 
\end{itemize}

For $\{i,j,k\}$ with $1 \le i < j < k \le n-1$, we have 
\[\Phi_3 ((ijkn))=\{i, j,k\} - \{i, j,n\} + \{i,k,n\} - \{j, k,n\}.\]
Hence we may take $w=(ijkn)$. 

For $\{i, j, n\}$ with $1 < i < j \le n-2$, we have 
\begin{align*}
\Phi_3 ((i (n-1) jn)) &= \{i, n-1, j\} - \{i, n-1, n\} + \{i, j, n\} - \{n-1, j, n\}, \\
\Phi_3 ((i j (n-1) n)) &= \{i,j, n-1\} - \{i, j, n\} + \{i, n-1,n\} - \{j, n-1, n\}. 
\end{align*}
 Then, 
 \[\Phi_3 ((i (n-1) jn)(i j (n-1) n)^{-1})=2\{i, j, n\} -2\{i, n-1, n\}.\]
 Hence we may take $w=(i (n-1) jn)(i j (n-1) n)^{-1}$. 
 
For $\{i, n-1, n\}$ with $1 \le i \le n-3$, we have 
\[\Phi_3 ((i (n-2)(n-1)n) (i (n-2) n (n-1)))=2\{i, n-1, n\} -2\{n-2, n-1, n\}.\]
Hence we may take $w=(i (n-2)(n-1)n) (i(n-2) n (n-1))$. 

From the above, our claim except that the image of $\Phi_3$ is not a direct summand 
readily follows by a usual argument in the theory of abelian groups.  
The remaining part will be proved in the next section (see Remark \ref{rem:notdirectsummand}).
\end{proof}

Theorem \ref{thm:generating_gammahat} gives an upper bound of 
the minimum number of generators of $\widehat{\Gamma_n^4}$ 
while Theorem \ref{thm:abelianization_gammahat} gives a lower bound. 
Since they coincide, we have the following. 
\begin{cor}\label{cor:minimalgen_hat}
For $n \ge 4$, the group $\widehat{\Gamma_n^4}$ needs 
$N_n=\begin{pmatrix} n \\ 3\end{pmatrix}-1$ elements to generate. 
The set $\Lambda$ in Theorem $\ref{thm:generating_gammahat}$ 
is a minimal generating set of $\widehat{\Gamma_n^4}$.
\end{cor}

\begin{remark}\label{rem:Phi2}
For $n \ge 5$, consider the homomorphism $\Phi_2 \colon \widehat{\Gamma_n^4}  \to \mathbb{Z} [n]_2$ 
defined by 
\[\Phi_2 ((ijkl)) = \{i, k\} - \{j, l\}.\]
Indeed our signed dihedral relation in $\widehat{\Gamma_n^4}$ was designed so that 
$\Phi_2$ is well-defined. 
Since 
\begin{align*}
&\Phi_2 ((i (n-1) j n)) =\{i, j\} -\{n-1, n\},\\
&\Phi_2 ((i j (n-1) k)(j (n-1)kn)) =\{i, n-1\} -\{n-1, n\},\\
&\Phi_2 ((i j n k)(j (n-1)kn)) =\{i, n\} -\{n-1, n\}
\end{align*}
for $i, j \le n-2$, we see that the image of $\Phi_2$ is precisely
\[\left\{\begin{array}{c|l}
\sum_{1 \le i < j \le n} a_{i,j} \{i, j\} \in \mathbb{Z}[n]_2 &  \sum_{1 \le i < j \le n} a_{i,j} =0
\end{array}\right\}.\]
The relationship between $\Phi_3$ and $\Phi_2$ is as follows. 
Define a homomorphism $\eta_3 \colon \mathbb{Z}[n]_3 \to \mathbb{Z}[n]_2$ by
\[\eta_3 (\{i,j,k\})=\{i,j\}+\{j,k\}+\{k,i\}.\]
Then it is easily checked that $2\Phi_2 = \eta_3 \circ \Phi_3$ holds. Hence, 
$\Phi_2$ does not have much information about $H_1 (\widehat{\Gamma_n^4})$ 
than $\Phi_3$. 
\end{remark}

\section{Minimal generating set of $\Gamma_n^4$}\label{sec:gamma}

Here we focus on the original groups $\Gamma_n^4$. 
When $n=4$, the group $\Gamma_4^4$ is generated by 
$\{(1234), (1324),(1243)\}$ and has only the involutive relations. 
Hence $\Gamma_4^4 \cong (\mathbb{Z}/2\mathbb{Z})^{\ast 3}$, the free product of three copies of $\Z/2\Z$. 

Let us give a minimal generating set of $\Gamma_n^4$ for general $n \ge 5$. 
Our proof of Theorem \ref{thm:generating_gammahat} 
is applicable words-by-words to $\Gamma_n^4$ after 
replacing $\widehat{\Gamma_n^4}$ by $\Gamma_n^4$. We have  
\begin{thm}\label{thm:generating_gamma}
For $n \ge 4$, the group $\Gamma_n^4$ is generated by the set $\Lambda$ 
consisting of 
\begin{itemize}
\item[(G1)] $(1 2 3 k)$ with $4 \le i \le n$, 

\item[(G2)] $(1 i 2 k)$ with $3 \le i < k \le n$, 

\item[(G3)] $(1 i j k)$ with $2 \le i < k < j \le n$. 
\end{itemize}
Therefore 
$\Gamma_n^4$ is also generated by 
$N_n=\displaystyle\frac{(n-3)(n^2+2)}{6}=
\begin{pmatrix} n \\ 3\end{pmatrix}-1$ elements. 
\end{thm}

Theorem \ref{thm:generating_gamma} and the involutive relation $(1)$ of $\Gamma_n^4$ say 
that there exists a  surjective homomorphism 
\[(\mathbb{Z}/2 \mathbb{Z})^{\ast N_n} \twoheadrightarrow \Gamma_n^4\]
for $n \ge 4$. Passing to their abelianizations, we have a surjective homomorphism 
\[(\mathbb{Z}/2\mathbb{Z})^{N_n} \twoheadrightarrow  H_1 (\Gamma_n^4).\]
We now see that the last surjection is an isomorphism. 
For that we use the homomorphisms $\Phi_3$ and $\Phi_2$ defined in the previous section. 
Note that as we see below only $\Phi_3$ does not suffice. 

Let $(\mathbb{Z}/2\mathbb{Z}) [n]_k$ denote 
the $(\mathbb{Z}/2\mathbb{Z})$-vector space based by the set $[n]_k$.  
The homomorphisms
\[\Phi_3 \otimes (\mathbb{Z}/2\mathbb{Z}) \colon \widehat{\Gamma_n^4} 
\longrightarrow (\mathbb{Z}/2\mathbb{Z}) [n]_3, \quad 
\Phi_2 \otimes (\mathbb{Z}/2\mathbb{Z}) \colon \widehat{\Gamma_n^4} 
\longrightarrow (\mathbb{Z}/2\mathbb{Z}) [n]_2\]
factor through $\Gamma_n^4$ and define the homomorphisms
\begin{align*}
&\Phi_3^{(2)} \colon \Gamma_n^4 \longrightarrow (\mathbb{Z}/2\mathbb{Z}) [n]_3, \quad 
\Phi_3^{(2)} ((ijkl)) = \{i, j, k\} + \{i, j, l\}+\{i,k,l\} + \{j, k, l\}, \\
&\Phi_2^{(2)} \colon \Gamma_n^4 \longrightarrow (\mathbb{Z}/2\mathbb{Z}) [n]_2, \quad 
\Phi_2^{(2)} ((ijkl)) = \{i, k\} + \{j, l\}. 
\end{align*}

\begin{thm}\label{thm:abelianization_gamma}
For $n \ge 4$, the image of the homomorphism 
\[\Phi_3^{(2)} \oplus \Phi_2^{(2)} \colon 
\Gamma_n^4 \longrightarrow (\mathbb{Z}/2\mathbb{Z}) [n]_3 \oplus 
(\mathbb{Z}/2\mathbb{Z}) [n]_2 \cong 
(\mathbb{Z}/2\mathbb{Z})^{\binom{n}{3}+\binom{n}{2}}\]
is isomorphic to $(\mathbb{Z}/2\mathbb{Z})^{N_n}$. 
Combining with Theorem $\ref{thm:generating_gamma}$, 
we have 
\[H_1 (\Gamma_n^4) \cong (\mathbb{Z}/2\mathbb{Z})^{N_n}.\]
\end{thm}
\begin{proof}
We first show that the image of $\Phi_3^{(2)}$ alone is isomorphic to 
$(\mathbb{Z}/2\mathbb{Z})^{\binom{n-1}{3}}$, 
which is smaller than $(\mathbb{Z}/2\mathbb{Z})^{N_n}$. 
For that we identify the image of $\Phi_3^{(2)}$ with that of the boundary map 
$\partial_3 \colon C_3 \to C_2$ of 
the simplicial chain complex $\{C_\ast, \partial_\ast\}$ with coefficients in $\mathbb{Z}/2\mathbb{Z}$ 
of an $(n-1)$-simplex, whose vertices are numbered $1,2,\ldots, n$. 
We have $C_k = (\mathbb{Z}/2\mathbb{Z}) [n]_{k+1}$. 
The chain complex 
\[\cdots \longrightarrow C_3 \stackrel{\partial_3}{\longrightarrow} 
C_2 \stackrel{\partial_2}{\longrightarrow} C_1 \stackrel{\partial_1}{\longrightarrow} C_0 
\stackrel{\varepsilon}{\longrightarrow}  \mathbb{Z}/2\mathbb{Z} \longrightarrow 0\]
 is known to be acyclic. Then we have 
 \begin{align*}
 \dim_{\mathbb{Z}/2\mathbb{Z}} \Im \Phi_3^{(2)} 
 &= \dim_{\mathbb{Z}/2\mathbb{Z}} \Im \partial_3 
 = \dim_{\mathbb{Z}/2\mathbb{Z}} \Ker \partial_2\\
 &= \dim_{\mathbb{Z}/2\mathbb{Z}} C_2 - \dim_{\mathbb{Z}/2\mathbb{Z}} \Im \partial_2
 =\binom{n}{3}- \dim_{\mathbb{Z}/2\mathbb{Z}} \Ker \partial_1\\
 &=\binom{n}{3}-(\dim_{\mathbb{Z}/2\mathbb{Z}} C_1-\dim_{\mathbb{Z}/2\mathbb{Z}} \Im \partial_1)
 =\binom{n}{3}- \binom{n}{2} + \dim_{\mathbb{Z}/2\mathbb{Z}} \Ker \varepsilon\\
 &=\binom{n}{3}-\binom{n}{2}+(n-1)= \frac{(n-1)(n-2)(n-3)}{6}= \binom{n-1}{3}. 
 \end{align*}
 
 Next we find the remaining $N_n- \displaystyle\binom{n-1}{3}=\displaystyle\frac{n (n-3)}{2}$ dimensional 
$ (\mathbb{Z}/2\mathbb{Z})$-vector space from $\Phi_2^{(2)}$. More precisely, 
we see that $\dim_{\mathbb{Z}/2\mathbb{Z}} 
\Phi_2^{(2)} \left(\Ker \Phi_3^{(2)}\right)$ is at least $\displaystyle\frac{n (n-3)}{2}$. 

We endow the basis $[n]_2$ of $ (\mathbb{Z}/2\mathbb{Z}) [n]_2$ 
with the lexicographic order $\prec$ after writing each element of $[n]_2$ in the form 
$\{i,j\}$ with $i>j$. That is, 
\[\{2,1\} \prec \{3,1\} \prec \{3,2\} \cdots \prec \{n, n-1\}.\]
By this total order, we regard  $[n]_2$ as an ordered basis of $ (\mathbb{Z}/2\mathbb{Z})[n]_2$. 
Let us show that for each $\{i,j\} \in [n]_2$ satisfying $\{i,j\} \prec \{n-1, n-2\}$ and 
$\{i,j\} \neq \{n-1, n-2\}$, there exists 
an element $w$ of $\Ker \Phi_3^{(2)}$ satisfying
\begin{itemize}
\item \ The coefficient of $\{i, j\}$ in $\Phi_2^{(2)} (w)$ is $1$,  
\item \ If $\{i', j'\} \prec \{i, j\}$ and $\{i', j'\} \neq \{i, j\}$, 
the coefficient of $\{i', j'\}$ in $\Phi_2^{(2)} (w)$ is zero. 
\end{itemize}

For $\{i,j\}$ with $1 \le j< i \le n-2$, we have 
\[\Phi_2^{(2)} ((i (n-1) j n)(i j (n-1) n))=\{i,j\} + \{n,n-1\} + \{n-1,i\} + \{n,j\}.\]
Hence we may take $w=(i (n-1) jn)(i j (n-1) n) \in \Ker \Phi_3^{(2)}$. 

For $\{n-1, j\}$ with $1 \le j \le n-3$, we have 
 \[\Phi_2^{(2)} ((j (n-2) (n-1) n)(j (n-2) n (n-1)))=\{n-1, j\} + \{n,n-2\} + \{n,j\} + \{n-1,n-2\}.\]
Hence we may take $w=(j (n-2) (n-1) n)(j (n-2) n (n-1)) \in \Ker \Phi_3^{(2)}$. 

From the above, we see that $\Phi_2 \left(\Ker \Phi_3^{(2)}\right)$ is at least 
$\displaystyle\binom{n-2}{2}+(n-3)=\displaystyle\frac{n (n-3)}{2}$ dimensional. 
This is what we want to show and we finish the proof. 
\end{proof}

Theorem \ref{thm:generating_gamma} gives an upper bound of 
the minimum number of generators of $\Gamma_n^4$  
while Theorem \ref{thm:abelianization_gamma} gives a lower bound. 
Since they coincide, we have the following. 
\begin{cor}\label{cor:minimalgen}
For $n \ge 4$, the group $\Gamma_n^4$ needs 
$N_n=\begin{pmatrix} n \\ 3\end{pmatrix}-1$ elements to generate. 
The set $\Lambda$ in Theorem $\ref{thm:generating_gamma}$ 
is a minimal generating set of $\Gamma_n^4$.
\end{cor}

\begin{remark}\label{rem:notdirectsummand}
As seen in the proof of Theorem \ref{thm:abelianization_gamma}, 
we have 
\[(\Phi_3 \otimes (\mathbb{Z}/2\mathbb{Z}))(\widehat{\Gamma_n^4})=\Phi_3^{(2)} (\Gamma_n^4) 
\cong (\mathbb{Z}/2\mathbb{Z})^{\binom{n-1}{3}}.\]
On the other hand, we saw in Theorem \ref{thm:abelianization_gammahat} that 
$\Phi_3 (\widehat{\Gamma_n^4}) \cong \mathbb{Z}^{\binom{n}{3}-1}$. 
Since $\displaystyle\binom{n}{3}-1 > \displaystyle\binom{n-1}{3}$ for $n \ge 4$, we see 
that $\Phi_3 (\widehat{\Gamma_n^4}) \subset \mathbb{Z} [n]_3$ is not a direct summand. 
\end{remark}

\section{The complex abelianization of $\widehat{\Gamma_n^4}$ 
as a representation of the symmetric group}\label{sec:rep}
We give a conceptually easier proof of 
$\Phi_3 (\widehat{\Gamma_n^4}) \cong \mathbb{Z}^{N_n}$ 
in Theorem \ref{thm:abelianization_gammahat} 
by using the representation theory of symmetric groups. 
We assume $n \ge 5$. 

By definition, we have a natural action of the symmetric group $\mathfrak{S}_n$ of degree $n$ on 
the group $\widehat{\Gamma_n^4}$, $\Gamma_n^4$ and also the set $[n]_k$. 
It is clear that the homomorphisms $\Phi_3$, $\Phi_2$, $\Phi_3^{(2)}$ and $\Phi_2^{(2)}$ 
are all $\mathfrak{S}_n$-equivariant. 

We now consider the complexified version of $\Phi_3$. It is given by 
\[\Phi_3^{\mathbb{C}} \colon H_1 (\widehat{\Gamma_n^4};\mathbb{C}) \longrightarrow \mathbb{C} [n]_3,\]
where $\mathbb{C} [n]_k$ denotes the $\mathbb{C}$-vector space based by $[n]_k$. 
To show that $\Phi_3 (\widehat{\Gamma_n^4}) \cong \mathbb{Z}^{N_n}$, 
it suffices to see that $\Im \Phi_3^\mathbb{C} \cong \mathbb{C}^{N_n}$. 
The map $\Phi_3^{\mathbb{C}}$ is an $\mathfrak{S}_n$-equivariant linear map, 
so that the image $\Im \Phi_3^\mathbb{C}$ is described 
in terms of representations of $\mathfrak{S}_n$. 

For generalities of the representation theory of 
$\mathfrak{S}_n$, we refer to the book Fulton-Harris \cite{FH}. 
The irreducible complex representations of $\mathfrak{S}_n$ are parametrized by the Young diagrams 
consisting of $n$ boxes. We use the standard notation $[n_1, n_2, \ldots, n_k]$ to denote a Young diagram 
where $n_1+n_2+ \cdots +n_k=n$ and $n_1 \ge n_2 \ge \cdots \ge n_k \ge 1$. 
We denote by $V_{[n_1, n_2, \ldots, n_k]}$ the corresponding representation space.   
It is known that the trivial one dimensional representation $\mathbb{C}=\mathbb{C}[n]_0$ 
corresponds to $V_{[n]}$ and the natural permutation action of $\mathfrak{S}_n$ on 
$\mathbb{C}^n=\mathbb{C}[n]_1$ gives the representation having 
the irreducible decomposition $\mathbb{C}[n]_1 = V_{[n]} \oplus V_{[n-1,1]}$. 
\begin{lem}\label{lem:irred}
For $n \ge 6$, we have $\mathfrak{S}_n$-irreducible decompositions 
\begin{align*}
\mathbb{C}[n]_2 &=V_{[n]} \oplus V_{[n-1,1]} \oplus V_{[n-2,2]},\\
\mathbb{C}[n]_3 &=V_{[n]} \oplus V_{[n-1,1]} \oplus V_{[n-2,2]} \oplus V_{[n-3,3]}.
\end{align*}
When $n=5$, we have $\mathbb{C}[5]_2=\mathbb{C}[5]_3= V_{[5]} \oplus V_{[4,1]} \oplus V_{[3,2]}$. 
\end{lem}
\begin{proof}
The authors guess that these decompositions are well-known. However they could not find 
a reference, so that we here give a brief proof. We check the characters of these representations. 
Recall that the character $\chi_\rho$ of a representation $\rho \colon \mathfrak{S}_n \to GL(V)$ 
is the function 
\[\chi_\rho \colon \mathfrak{S}_n / \mathrm{conjugate} \longrightarrow \mathbb{C}, 
\qquad \chi_\rho ([\sigma])= \mathrm{Tr} (\rho (\sigma)),\]
where $[\sigma]$ is the conjugacy class of an element $\sigma \in \mathfrak{S}_n$. 
By considering the usual decomposition of an element $\sigma \in \mathfrak{S}_n$ 
into the product of cyclic permutations, the conjugacy classes of $\mathfrak{S}_n$ has 
the one-to-one correspondence with the Young diagrams of $n$ boxes.

Let us compute the character $\chi_k$ of $\mathbb{C}[n]_k$ for $k=2, 3$. 
Since the action of $\mathfrak{S}_n$ on $\mathbb{C}[n]_k$ is given by permutations of 
the basis, it suffices to consider the number of fixed points for our computation. 

Let $C_i$ be the conjugacy class of an element $\sigma \in \mathfrak{S}_n$ having 
$i_k$ cyclic permutations of length $k$. 
Then we have 
\[\chi_2 (C_i)=i_2 + \binom{i_1}{2}.\]
To explain this formula, let us consider the case $n=8$ and $\sigma=(123)(45)(6)(7)(8)$. 
The fixed points of $\sigma$ are given by 
\begin{itemize}
\item $\{4,5\}$, where $\sigma$ exchanges $4$ and $5$, 
\item $\{6,7\}, \{7,8\}, \{6,8\}$, where $\sigma$ fixes each element of these subsets. 
\end{itemize}
In general, the fixed points in $[n]_2$ are obtained from these patterns. 
By a consideration similar to the above, we have 
\[\chi_3 (C_i)=i_3 + i_1 i_2 + \binom{i_1}{3}.\]

On the other hand, we may compute the character $\chi_{[n_1, n_2, \ldots, n_k]}$ of 
the irreducible representation $V_{[n_1, n_2, \ldots, n_k]}$ of $\mathfrak{S}_n$ by 
the Frobenius character formula. We have 
\begin{align*}
&\chi_{[n]} (C_i)=1,\\
&\chi_{[n-1,1]} (C_i)=i_1-1,\\
&\chi_{[n-2,2]} (C_i)=i_2 +\frac{i_1 (i_1-3)}{2},\\
&\chi_{[n-3,3]} (C_i)=i_3+i_2 (i_1-1) + \binom{i_1}{3}-\binom{i_1}{2}.
\end{align*}
Now it is easy to see that 
\begin{align*}
\chi_2 (C_i)&=\chi_{[n]} (C_i) + \chi_{[n-1,1]} (C_i) +\chi_{[n-2,2]} (C_i),\\
\chi_3 (C_i)&=\chi_{[n]} (C_i) + \chi_{[n-1,1]} (C_i) +\chi_{[n-2,2]} (C_i)+\chi_{[n-3,3]} (C_i)
\end{align*}
hold for any $C_i \in \mathfrak{S}_n / \mathrm{conjugate}$. 
The desired irreducible decompositions follow from these, 
since finite dimensional representations of $\mathfrak{S}_n$ 
are characterized by their characters. Note that 
we have $\chi_{[n-3,3]} (C_i)\equiv 0$ when $n=5$, so we need to omit this term. 
\end{proof}

\begin{lem}\label{lem:irred2}
For $n \ge 6$, we have $\Im \Phi_3^{\mathbb{C}} = V_{[n-1,1]} \oplus V_{[n-2,2]} \oplus V_{[n-3,3]}$ as 
$\mathfrak{S}_n$-irreducible decompositions. 
When $n=5$, we have $\Im \Phi_3^{\mathbb{C}}=\Im \Phi_2^{\mathbb{C}}=V_{[4,1]} \oplus V_{[3,2]}$. 
\end{lem}
\begin{proof}
From a computation in Remark \ref{rem:Phi2}, the cokernel of 
the complexified version 
\[\Phi_2^{\mathbb{C}} \colon 
H_1 (\widehat{\Gamma_n^4};\mathbb{C}) \to \mathbb{C} [n]_2\]
of $\Phi_2$ is one dimensional and corresponds to $V_{[n]} \subset \mathbb{C}[n]_2$. Hence 
\[\Im \Phi_2^{\mathbb{C}}=V_{[n-1,1]} \oplus V_{[n-2,2]}.\]
Since $\Phi_2^{\mathbb{C}}=(\eta_3 \otimes \mathbb{C}) \circ \Phi_3^{\mathbb{C}}$, 
we see that $\Im \Phi_3^{\mathbb{C}}$ includes $V_{[n-1,1]} \oplus V_{[n-2,2]}$. 
By the hook length formula, we have 
$\dim V_{[n-1,1]}=n-1$ and $\dim V_{[n-2,2]}=\displaystyle\frac{n(n-3)}{2}$. 
When $n=5$, we have done since $\dim V_{[4,1]}+\dim V_{[3,2]}=9=N_5$. 
When $n \ge 6$, we have 
\begin{align*}
\Phi_3^{\mathbb{C}} ((1324)(3546)(5162)) &=1 \cdot \{1,2,3\}+0 \cdot \{1,3,5\}+ \cdots \neq 0, \\
\Phi_2^{\mathbb{C}} ((1324)(3546)(5162)) &=0, \\
\Phi_3^{\mathbb{C}} ((1325)(3456)(4162)) &=1 \cdot \{1,2,3\}-1 \cdot \{1,3,5\}+ \cdots \neq 0,\\
\Phi_2^{\mathbb{C}} ((1325)(3456)(4162)) &=0.
\end{align*}
These equalities imply that $\Im \Phi_3^{\mathbb{C}} \cap \Ker (\eta_3 \otimes \mathbb{C})$ is 
at least $2$-dimensional. Then we see from the irreducible decomposition of $\mathbb{C}[n]_3$ in Lemma \ref{lem:irred} 
that $\Im \Phi_3^{\mathbb{C}}$ includes $V_{[n-3,3]}$. 
Since $\dim V_{[n-3,3]}=\displaystyle\frac{n(n-1)(n-5)}{6}$ and 
$\dim V_{[n-1,1]}+\dim V_{[n-2,2]}+\dim V_{[n-3,3]} = N_n$, we complete the proof.
\end{proof}

Combining Corollary \ref{cor:minimalgen_hat} and Lemma \ref{lem:irred2}, we have the following. 
\begin{thm}\label{thm:H1irred}
For $n \ge 6$, we have an $\mathfrak{S}_n$-irreducible decomposition
\[H_1 (\widehat{\Gamma_n^4};\mathbb{C}) = 
V_{[n-1,1]} \oplus V_{[n-2,2]} \oplus V_{[n-3,3]}.\]
When $n=5$, we have $H_1 (\widehat{\Gamma_5^4};\mathbb{C}) 
= V_{[4,1]} \oplus V_{[3,2]}$. 
\end{thm}

\section{The infiniteness of $\Gamma_5^4$}\label{sec:n5}
\subsection{GAP computation}\label{subsec:GAP}
To see further structures of $\Gamma_n^4$ beyond the abelianization, we may use the program GAP. 
Here we report some results for $\Gamma_5^4$ obtained by GAP computations. 

After inputing a presentation for $\Gamma_5^4$, 
we may use the command ``DerivedSubgroup'' to get the data of the commutator subgroup 
$[\Gamma_5^4, \Gamma_5^4]$. Then the command ``AbelianInvariants'' 
computes $H_1 ([\Gamma_5^4, \Gamma_5^4])$. The result is 
\[H_1 ([\Gamma_5^4, \Gamma_5^4]) \cong \mathbb{Z}^{145} \oplus (\mathbb{Z}/2\mathbb{Z})^{18}.\]
From this we readily have the following. 

\begin{thm}\label{thm:GAP}
The group $\Gamma_5^4$ is an infinite non-commutative group. 
Moreover it does not have Property $(\mathrm{T})$.
\end{thm}

\begin{proof}
Since $[\Gamma_5^4, \Gamma_5^4]$ is an infinite group, we immediately see that 
$\Gamma_5^4$ is infinite and non-commutative. 

By Theorem \ref{thm:abelianization_gamma}, 
the group $[\Gamma_5^4, \Gamma_5^4]$ is a finite index subgroup of $\Gamma_5^4$. 
The above GAP computation says that the abelianization of $[\Gamma_5^4, \Gamma_5^4]$ 
has a $\mathbb{Z}$-summand. Then by a general fact on Property (T), 
we see that $\Gamma_5^4$ does not have it. 
\end{proof}

\subsection{Proving the infiniteness of $\Gamma_5^4$ by hand}\label{subsec:RS}

We give another proof of Theorem \ref{thm:GAP}, which does not use GAP. 
We see that the group $\Gamma_5^4$ has a subgroup of index $2$ 
having a $\mathbb{Z}$-summand in its abelianization. 

Let us simplify the defining presentation of $\Gamma_5^4$. 
Note that there are no commutative relations when $n=5$. 
We now apply Tietze transformations in order. 
First, we remark that under the involutive relations and the dihedral relations, 
the pentagon relations have a dihedral symmetry. 
Indeed, the left hand side of the pentagon relation 
\[(ijkl)(ijlm)(jklm)(ijkm)(iklm)=1\]
for $\{i,j,k,l,m\}$ is rewritten by shifting the word cyclically to the left twice 
and applying the dihedral relation to 
\[(jklm)(jkmi)(klmi)(jkli)(jlmi)=1, \]
which is the pentagon relation for $\{j,k,l,m,i\}$. Also, taking the inverse of 
the left hand side of the pentagon relation for $\{j,k,l,m,i\}$, we have 
\[(jlmi)(jkli)(klmi)(jkmi)(jklm)=1.\]
We apply the dihedral relation to the left hand side and shift the word cyclically  
to the right once. Then we get 
\[(mlkj)(mlji)(lkji)(mlki)(mkji)=1,\]
which is the pentagon relation for $\{m,l,k,j,i\}$. 
Using this symmetry and dihedral relations (to the underlined parts), 
we may reduce the pentagon relations to the following $5!/10=12$ relations: 
\begin{align*}
(a):  \ \  &(1234)(1245)(2345)(1235)(1345)=1, \quad \text{for $\{1,2,3,4,5\}$,}\\
(b):  \ \  &(1235)(1254)(2354)(1234)(1354)=1, \quad \text{for $\{1,2,3,5,4\}$,}\\
(c):  \ \  &(1243)(1235)(2435)(1245)(1435)=1, \quad \text{for $\{1,2,4,3,5\}$,}\\
(d):  \ \  &(1245)(1253)\underline{(2354)}(1243)\underline{(1354)}=1, \quad \text{for $\{1,2,4,5,3\}$,}\\
(e):  \ \  &(1253)(1234)\underline{(2435)}(1254)\underline{(1435)}=1, \quad \text{for $\{1,2,5,3,4\}$,}\\
(f):  \ \  &(1254)(1243)\underline{(2345)}(1253)\underline{(1345)}=1, \quad \text{for $\{1,2,5,4,3\}$,}\\
(g):  \ \  &(1324)(1345)\underline{(2354)}(1325)(1245)=1, \quad \text{for $\{1,3,2,4,5\}$,}\\
(h):  \ \  &(1325)(1354)\underline{(2345)}(1324)(1254)=1, \quad \text{for $\{1,3,2,5,4\}$,}\\
(i):  \ \  &\underline{(1324)}(1435)\underline{(2354)}(1425)(1235)=1, \quad \text{for $\{1,4,2,3,5\}$,}\\
(j):  \ \  &(1425)\underline{(1354)}\underline{(2435)}\underline{(1324)}(1253)=1, \quad \text{for $\{1,4,2,5,3\}$,}\\
(k):  \ \  &\underline{(1325)}\underline{(1435)}\underline{(2345)}\underline{(1425)}(1234)=1, 
\quad \text{for $\{1,5,2,3,4\}$,}\\
(l):  \ \  &\underline{(1425)}\underline{(1345)}\underline{(2435)}\underline{(1325)}(1243)=1, \quad \text{for $\{1,5,2,4,3\}$.}\\
\end{align*}
Then we use the dihedral relations to reduce the generating set to 
the following set $\Theta$ consisting of $15$ elements: 
\[\Theta:=\left\{\begin{array}{l}
(1234), (1235), (1243), (1245), (1253), (1254),\\ 
(1324), (1325), (1345), (1354), (1425), (1435),\\
(2345), (2354), (2435)
\end{array}\right\}\]
and erase the involutive relations for the discarded generators. 
Consequently, we have a presentation $P=\langle \Theta \mid R\rangle$ 
of $\Gamma_5^4$ consisting of $15$ generators and 
$27$ relations. 

For this presentation, consider the map $\nu \colon \Theta \to \mathbb{Z}/2\mathbb{Z}$ defined by 
\[\nu ((ijkl))=\begin{cases} 1 & \text{(if $1 \in \{i,j,k,l\}$)}\\ 0 & \text{(otherwise)}\end{cases}.\]
It extends to a well-defined homomorphism $\nu \colon \Gamma_5^4 \to \mathbb{Z}/2\mathbb{Z}$. 

\begin{thm}\label{thm:RS}
The abelianization of the kernel of the homomorphism $\nu$ has a $\mathbb{Z}$-summand. 
\end{thm}

\begin{proof}
We obtain a presentation of $\Ker \nu$ by applying the Reidemeister-Schreier method 
to the above presentation $P$ and abelianize it. 
We refer to the book Magnus-Karrass-Solitar \cite[Section 2.3]{MKS} for the details on 
the Reidemeister-Schreier method. 

For the presentation $P$, we may take $T:=\{1, (1234)\}$ as a set of Schreier transversals. 
Then the Reidemeister-Schreier method says that $\Ker \nu$ is generated by 
the set $\{ t x \overline{tx}^{-1} \in \Ker \nu \mid t \in T, x \in \Theta\}$, where for $y \in \Gamma_5^4$, 
we have $\overline{y}=1$ if $y \in \Ker \nu$ and $\overline{y}=(1234)$ otherwise. 
Explicitly, we have the following 
generators of $\Ker \nu$: 
\begin{align*}
&\alpha (1234):=(1234) (1234)^{-1},\\
&\alpha (1235):=(1235) (1234)^{-1}, \ \ \alpha (1243):=(1243) (1234)^{-1}, \ 
\ldots \ , \ \ \alpha (1435):=(1435)(1234)^{-1},\\
&\alpha (2345):=(2345), \ \ \alpha (2354):=(2354), \ \ \alpha (2435):=(2435),\\
&\beta (1234):=(1234)(1234), \\
&\beta (1235):=(1234) (1235), \ \ \beta (1243):=(1234) (1243),  \ 
\ldots, \ \beta (1435):=(1234)(1435),\\
&\beta (2345):=(1234)(2345)(1234)^{-1}, \ \ \beta (2354):=(1234)(2354)(1234)^{-1},\\
&\beta (2435):=(1234)(2435)(1234)^{-1}.
\end{align*}
The relations consist of two types (see \cite[Theorem 2.9]{MKS}). The first type is $\alpha (1234)=1$ since 
$\alpha (1234)=(1234)(1234)^{-1}$ is freely equal to the trivial element. 
The second type is 
\[\{\tau (t r t^{-1})=1 \mid t \in T, r \in R\},\] 
where $\tau$ is the Reidemeister-Schreier rewriting process for $T$. 
Explicitly the relations of the second type are given as follows: 

\medskip
\noindent
(From the involutive relations)
\begin{align*}
&\tau((1234)(1234))=\alpha (1234)\beta (1234)=1,\\
&\tau((1235)(1235))=\alpha (1235)\beta (1235)=1,\\
& \hskip 87pt \vdots\\
&\tau((1435)(1435))=\alpha (1435)\beta (1435)=1,\\
&\tau((2345)(2345))=\alpha (2345)\alpha (2345)=1,\\
&\tau((2354)(2354))=\alpha (2354)\alpha (2354)=1,\\
&\tau((2435)(2435))=\alpha (2435)\alpha (2435)=1,\\
& \\
&\tau((1234)(1234)(1234)(1234)^{-1})=\alpha (1234)\beta (1234)=1,\\
&\tau((1234)(1235)(1235)(1234)^{-1})=\alpha (1234)\beta (1235)\alpha (1235)\alpha (1234)^{-1}=1,\\
& \hskip 163pt \vdots\\
&\tau((1234)(1435)(1435)(1234)^{-1})=\alpha (1234)\beta (1435)\alpha (1435)\alpha (1234)^{-1}=1,\\
&\tau((1234)(2345)(2345)(1234)^{-1})=\alpha (1234)\beta (2345)\beta (2345)\alpha (1234)^{-1}=1,\\
&\tau((1234)(2354)(2354)(1234)^{-1})=\alpha (1234)\beta (2354)\beta (2354)\alpha (1234)^{-1}=1,\\
&\tau((1234)(2435)(2435)(1234)^{-1})=\alpha (1234)\beta (2435)\beta (2435)\alpha (1234)^{-1}=1,\\
\end{align*}

\medskip
\noindent
(From the pentagon relations) \ \ We simply write $(x)$ for the left hand side of 
the $12$ pentagon relations mentioned before, where $x=a,b,c,\ldots, l$. 

\begin{align*}
&\tau((a))=\alpha (1234)\beta (1245)\alpha (2345)\alpha (1235)\beta (1345)=1,\\
&\tau((b))=\alpha (1235)\beta (1254)\alpha (2354)\alpha (1234)\beta (1354)=1,\\
&\tau((c))=\alpha (1243)\beta (1235)\alpha (2435)\alpha (1245)\beta (1435)=1,\\
&\tau((d))=\alpha (1245)\beta (1253)\alpha (2354)\alpha (1243)\beta (1354)=1,\\
&\tau((e))=\alpha (1253)\beta (1234)\alpha (2435)\alpha (1254)\beta (1435)=1,\\
&\tau((f))=\alpha (1254)\beta (1243)\alpha (2345)\alpha (1253)\beta (1345)=1,\\
&\tau((g))=\alpha (1324)\beta (1345)\alpha (2354)\alpha (1325)\beta (1245)=1,\\
&\tau((h))=\alpha (1325)\beta (1354)\alpha (2345)\alpha (1324)\beta (1254)=1,\\
&\tau((i))=\alpha (1324)\beta (1435)\alpha (2354)\alpha (1425)\beta (1235)=1,\\
&\tau((j))=\alpha (1425)\beta (1354)\alpha (2435)\alpha (1324)\beta (1253)=1,\\
&\tau((k))=\alpha (1325)\beta (1435)\alpha (2345)\alpha (1425)\beta (1234)=1,\\
&\tau((l))=\alpha (1425)\beta (1345)\alpha (2435)\alpha (1325)\beta (1243)=1,
\end{align*}
\begin{align*}
&\tau((1234)(a)(1234)^{-1})=\alpha (1234)\beta (1234)\alpha (1245)\beta (2345)\beta (1235)\alpha (1345)\beta (1234)^{-1}=1,\\
&\tau((1234)(b)(1234)^{-1})=\alpha (1234)\beta (1235)\alpha (1254)\beta (2354)\beta (1234)\alpha (1354)\beta (1234)^{-1}=1,\\
&\tau((1234)(c)(1234)^{-1})=\alpha (1234)\beta (1243)\alpha (1235)\beta (2435)\beta (1245)\alpha (1435)\beta (1234)^{-1}=1,\\
&\tau((1234)(d)(1234)^{-1})=\alpha (1234)\beta (1245)\alpha (1253)\beta (2354)\beta (1243)\alpha (1354)\beta (1234)^{-1}=1,\\
&\tau((1234)(e)(1234)^{-1})=\alpha (1234)\beta (1253)\alpha (1234)\beta (2435)\beta (1254)\alpha (1435)\beta (1234)^{-1}=1,\\
&\tau((1234)(f)(1234)^{-1})=\alpha (1234)\beta (1254)\alpha (1243)\beta (2345)\beta (1253)\alpha (1345)\beta (1234)^{-1}=1,\\
&\tau((1234)(g)(1234)^{-1})=\alpha (1234)\beta (1324)\alpha (1345)\beta (2354)\beta (1325)\alpha (1245)\beta (1234)^{-1}=1,\\
&\tau((1234)(h)(1234)^{-1})=\alpha (1234)\beta (1325)\alpha (1354)\beta (2345)\beta (1324)\alpha (1254)\beta (1234)^{-1}=1,\\
&\tau((1234)(i)(1234)^{-1})=\alpha (1234)\beta (1324)\alpha (1435)\beta (2354)\beta (1425)\alpha (1235)\beta (1234)^{-1}=1,\\
&\tau((1234)(j)(1234)^{-1})=\alpha (1234)\beta (1425)\alpha (1354)\beta (2435)\beta (1324)\alpha (1253)\beta (1234)^{-1}=1,\\
&\tau((1234)(k)(1234)^{-1})=\alpha (1234)\beta (1325)\alpha (1435)\beta (2345)\beta (1425)\alpha (1234)\beta (1234)^{-1}=1,\\
&\tau((1234)(l)(1234)^{-1})=\alpha (1234)\beta (1425)\alpha (1345)\beta (2435)\beta (1325)\alpha (1243)\beta (1234)^{-1}=1. 
\end{align*}

By the relations from the involutive relations, we may erase $\alpha (1234)$ and $\beta (1234)$. 
Also we may erase $\beta (1jkl)$ since it equals to $\alpha (1jkl)^{-1}$. 
We have $6$ involutive relations $\alpha (2jkl)^2=\beta (2jkl)^2=1$. These simplify the presentation, so that 
the resulting one has $17$ generators ($14$ generators $\alpha (ijkl)$ and $3$ generators $\beta (2jkl)$) 
and $30$ relations ($6$ involutive relations and $24$ relations above). 
We can write a presentation matrix for $H_1 (\Ker \nu)$ and compute its Smith normal form. 
Here we omit the details since it is a usual matrix computation. The result is 
$H_1 (\Ker \nu) \cong \mathbb{Z}^2 \oplus (\mathbb{Z}/2\mathbb{Z})^6$ with 
the generators of $\mathbb{Z}^2$ given by $\alpha (1324)=(1324)(1234)$ and $\alpha (1425)=(1425)(1234)$, 
which completes the proof. 
\end{proof}

\section{Increasing-order version of $\Gamma_n^4$}\label{sec:delta}

Finally, we introduce new groups $\Delta_n^4$, which look simpler than $\Gamma_n^4$. 
They might be helpful to investigate the structure of $\Gamma_n^4$. 

\begin{definition}\label{def:delta}
For $n \ge 4$, the group $\Delta_n^4$ is defined by the following presentation: 

(Generators) $\{(ijkl) \mid 1 \le i < j < k < l \le n\}$

(Relations) There are three types of relations: 
\[\begin{array}{ll}
 (1)&  (ijkl)^2 =1; \\
 (2)&  (ijkl)(stuv)=(stuv)(ijkl),\ (|\{i,j,k,l\} \cap \{s,t,u,v\}| \le 2);\\
 (3)&  (ijkl)(ijlm)(jklm)(ijkm)(iklm)=1, \ (1 \le i < j < k < l < m \le n).\\
\end{array}\]
\end{definition}

Note that we have the natural homomorphism $\Delta_n^4 \to \Gamma_n^4$ sending 
$(ijkl) \in \Delta_n^4$ to $(ijkl) \in \Gamma_n^4$. When $n=4$, the group $\Delta_4^4$ is given by 
\[\Delta_4^4=\langle (1234) \mid (1234)^2=1\rangle \cong \mathbb{Z}/2\mathbb{Z}.\]
We will discuss $\Delta_n^4$ in a way similar to the one for $\Gamma_n^4$ in previous sections.  

\begin{thm}\label{thm:delta}
For $n \ge 4$, we have the following. 

$(1)$ $\Delta_n^4$ needs $N'_n:=\displaystyle\binom{n-1}{3}=\displaystyle\frac{(n-1)(n-2)(n-3)}{6}$ elements to generate.

$(2)$ The set $\Lambda':=\{(1jkl) \mid 2 \le j < k < l \le n\}$ is a minimal generating set of $\Delta_n^4$. 

$(3)$ $H_1 (\Delta_n^4) \cong (\mathbb{Z}/2\mathbb{Z})^{N'_n}$. 
\end{thm}
\begin{proof}
The case where $n=4$ is clear. We assume that $n \ge 5$. The pentagon relation for $\{1,i,j,k,l\}$ says that 
\[(ijkl)=(1ikl)(1ijk)(1jkl)(1ijl),\]
which shows that $\Lambda'$ is a generating set. We now consider the homomorphism 
\[\Phi_3^{(2)} \colon \Delta_n^4 \longrightarrow (\mathbb{Z}/2\mathbb{Z}) [n]_3\]
defined by the same formula as $\Phi_3^{(2)}$ for $\Gamma_n^4$. 
We have 
\[\Phi_3^{(2)}((1jkl))=\{1,j,k\} + \{1,j,l\} + \{1,k,l\} +\{j,k,l\}.\] 
It is easy to prove our assertions from this equality. 
\end{proof}

\begin{remark}\label{rem:deltahat}
As in Section \ref{sec:gammahat}, we may define the ``hat version'' $\widehat{\Delta_n^4}$ 
of $\Delta_n^4$. 
The argument in Theorem \ref{thm:delta} is applicable to $\widehat{\Delta_n^4}$ 
almost word-by-word and we can get a similar statement. 
\end{remark}

When $n=5$, the presentation of $\Delta_5^4$ is rewritten as 
\[\Delta_5^4=\left\langle \begin{array}{c|l}
(1245),(1234),(1345),(1235) & \begin{array}{l}
(1245)^2=(1234)^2=(1345)^2=(1235)^2=1\\
((1245)(1234)(1345)(1235))^2=1
\end{array}\end{array}\right\rangle.  \]

\begin{thm}\label{thm:delta5}
The group $\Delta_5^4$ has a subgroup $G$ of index $2$ with 
$H_1 (G) \cong \mathbb{Z}^2 \oplus (\mathbb{Z}/2\mathbb{Z})$. 
Therefore $\Delta_5^4$ is an infinite non-commutative group and 
it does not have Property $(\mathrm{T})$.
\end{thm}
\begin{proof}
We construct a cell complex $X$ with $\pi_1 (X) \cong \Delta_5^4$ and its 
double cover $Y$. Take the usual cell decomposition of the $2$-sphere $S^2$ 
having two $0$-cells $(0,0,1)$ and $(0,0,-1)$. We write $\gamma$ for 
the path in $S^2$ given by $\gamma (t)=(0,\sin (\pi t), \cos (\pi t))$, which gives 
one of the two $1$-cells of the cell decomposition. 

Let $Y_1$ be the cell complex obtained from the 
disjoint union of four copies $S_1$, $S_2$, $S_3$, $S_4$ of the $2$-sphere $S^2$ 
by identifying the four points $(0,0,1)$ (resp.\ $(0,0,-1)$) in $S_i$ with $i=1,2,3,4$. 
We denote the identified point by $p_N$ (resp. $p_S$). 
Let $\gamma_i$ be the copy of $\gamma$ in $S_i \subset Y_1$ for $i=1,2,3,4$. 
It goes from $p_N$ to $p_S$. We write $\overline{\gamma_i}$ for the inverse path of $\gamma_i$. 
Then 
\[a=\gamma_1 \overline{\gamma_2}, \qquad b=\gamma_2 \overline{\gamma_3}, 
\qquad c=\gamma_3 \overline{\gamma_4}\] 
are loops generating $\pi_1 (Y_1,p_N) \cong \mathbb{Z}^{\ast 3}$, a free group of rank $3$. 
We attach to $Y_1$ two $2$-cells $e_1^2$ and $e_2^2$ as follows: 
$e_1^2$ is attached along the word $acac$ and 
$e_2^2$ is attached along the loop 
$\overline{\gamma_1} \gamma_2 \overline{\gamma_3} \gamma_4 
\overline{\gamma_1} \gamma_2 \overline{\gamma_3} \gamma_4$. 
We denote the resulting cell complex by $Y$. 
Define the free involution $\iota$ of $Y$ so that 
$\iota |_{S_i}$ is the antipodal map of the $2$-sphere $S_i$ and 
$\iota$ exchanges $e_1^2$ and $e_2^2$ naturally. 
Let $X$ be the quotient complex $Y/\iota$ and 
$q \colon Y \to X$ be the natural projection. 
The cell complex $X$ is obtained from 
$X_1:=Y_1/\iota$ by attaching a $2$-cell $e^2$ along the loop $q(acac)$. 
Now $X_1$ is homeomorphic to the bouquet $\displaystyle\bigvee_{i=1}^4 \mathbb{R}P^2$ of 
four copies of the real projective plane $\mathbb{R}P^2$. 
The paths $q(\gamma_i)$ in $X$ 
are loops generating $\pi_1 (X_1,q (p_N)) \cong (\mathbb{Z}/2\mathbb{Z})^{\ast 4}$. 
We name the four generators $(1245), (1234), (1345), (1235)$. 
The $2$-cell $e^2$ is attached along $((1245)(1234)(1345)(1235))^2$. Hence 
$\pi_1 (X,q(p_N)) \cong \Delta_5^4$ and $Y$ is a double cover of $X$. 
From the construction of $Y$, it is easy to see that 
\[\pi_1 (Y,p_N) \cong \langle a, b, c \mid acac=1, bc^{-1}b^{-1}a^{-1} bc^{-1}b^{-1}a^{-1}=1 \rangle,\]
Here, the loop $bc^{-1}b^{-1}a^{-1} bc^{-1}b^{-1}a^{-1}$ is freely homotopic to the loop 
$\overline{\gamma_1} \gamma_2 \overline{\gamma_3} \gamma_4 
\overline{\gamma_1} \gamma_2 \overline{\gamma_3} \gamma_4$. 
Then we have $H_1 (Y) \cong \mathbb{Z}^2 \oplus (\mathbb{Z}/2\mathbb{Z})$ generated by 
$a, b, a+c$ with the relation $2(a+c)=0$. We may take $\pi_1 (Y,p_N)$ as $G$. 
\end{proof}

\begin{remark}
When $n=6$, we checked with a help of GAP computations that 
the natural homomorphism $\Delta_6^4 \to \Gamma_6^4$ is not injective. Details will be 
discussed in a forthcoming paper. 
\end{remark}

\noindent
{\it Acknowledgement} \quad 
The authors would like to thank Ken'ichi Ohshika, Sumio Yamada and Nariya Kawazumi 
for giving the authors a chance to have a presentation of this work in the conference. 
Also, they would like to thank the referee for helpful comments to improve the paper.

\bibliographystyle{amsplain}

\end{document}